\documentclass[12pt]{amsart}

\usepackage{graphicx, amsmath, amscd, hyperref}

\newtheorem{thm}{Theorem}[section]
\newtheorem{lemma}[thm]{Lemma}

\newtheorem{cor}[thm]{Corollary}

\newtheorem*{lemma:lemma_main}{Lemma \ref{lemma_main}}
\newtheorem*{thm:thm_main}{Theorem \ref{thm_main}}
\newtheorem*{cor:cor_main}{Corollary \ref{cor_main}}
\newtheorem*{thm:thm_also_1}{Theorem \ref{thm_also_1}}
\newtheorem*{thm:thm_also_2}{Theorem \ref{thm_also_2}}
\newtheorem*{cor:cor_also}{Corollary \ref{cor_also}}

\theoremstyle{definition}

\newtheorem{rem}[thm]{Remark}
\newtheorem{question}[thm]{Question}

\newcommand{\diam}{\mathop{\rm diam}}
\newcommand{\Homeo}{\mathop{\rm Homeo}}
\newcommand{\Diff}{\mathop{\rm Diff}}
\newcommand{\id}{\mathop{\rm id}}

\newcommand{\Aut}{\mathop{\rm Aut}}

\begin{document}

\author{Michael P. Cohen}
\address{Michael P. Cohen,
Department of Mathematics,
North Dakota State University,
PO Box 6050,
Fargo, ND, 58108-6050}
\email{michael.cohen@ndsu.edu}

\title{On the large-scale geometry of diffeomorphism groups of $1$-manifolds}

\begin{abstract}  We apply the framework of Rosendal to study the large-scale geometry of the topological groups $\Diff_+^k(M^1)$, consisting of orientation-preserving $C^k$-diffeomorphisms (for $1\leq k\leq\infty$) of a compact $1$-manifold $M^1$ ($=I$ or $\mathbb{S}^1$).  We characterize the relative property (OB) in such groups: $A\subseteq\Diff_+^k(M^1)$ has property (OB) relative to $\Diff_+^k(M^1)$ if and only if $\displaystyle\sup_{f\in A}\sup_{x\in M^1}|\log f'(x)|<\infty$ and $\displaystyle\sup_{f\in A}\sup_{x\in M^1}|f^{(j)}(x)|<\infty$ for every integer $2\leq j\leq k$.  We deduce that $\Diff_+^k(M^1)$ has the local property (OB), and consequently a well-defined non-trivial quasi-isometry class, if and only if $k<\infty$.  We show that the groups $\Diff_+^1(I)$ and $\Diff_+^1(\mathbb{S}^1)$ are quasi-isometric to the infinite-dimensional Banach space $C[0,1]$.
\end{abstract}

\maketitle


\section{Introduction}

The purpose of this note is to study the large-scale geometry (or quasi-isometry type) of the topological groups $\Diff_+^k(I)$ and $\Diff_+^k(\mathbb{S}^1)$ of orientation-preserving $C^k$-diffeomorphisms of the interval $I$ and the circle $\mathbb{S}^1$, for $1\leq k\leq\infty$, and to concretely identify this geometry in case $k=1$.

Some clarification is needed since the study of quasi-isometry types of groups is most commonly confined to the case of finitely generated countable groups, while the diffeomorphism groups above are uncountable, hence certainly not finitely generated.  We work in the very general framework for coarse geometry recently advanced by Rosendal, which in many ways gives a satisfactory unification of active branches of research in large scale geometry of countable finitely generated groups (the domain of classical geometric group theory), the large scale geometry of locally compact groups (see the in-progress \cite{decornulier_delaharpe_2015a} for a thorough survey), and the coarse geometry of Banach spaces (see \cite{kalton_2008a}, \cite{ostrovskii_2013a}, for instance).  We give a brief summary of this general framework below, and refer the reader to \cite{rosendal_2015a} for a more detailed account.

\subsection{Context For Non-Locally Compact Groups}

If $G$ is a finitely generated countable group, then the quasi-isometry type of $G$ is determined by its finite symmetric generating sets $\Sigma$.  Each such $\Sigma$ defines an associated (right) Cayley metric on $G$, and each of these Cayley metrics are mutually quasi-isometric.  So any choice of Cayley metric may be said to be a representative of a canonical quasi-isometry class for $G$, and therefore to represent the large-scale geometry of $G$.  To extend this theory to groups $G$ which are perhaps not countable or finitely generated, one must seek conditions on possible generating sets $\Sigma\subseteq G$, which will guarantee that the associated Cayley metrics will turn out to be quasi-isometric.

The theory of quasi-isometry types extends readily if one assumes $G$ to be a locally compact, compactly generated metrizable topological group.  In this case, compact sets play the analogue of finite sets in countable groups, because any Cayley metric defined using a compact generating set $\Sigma$ for $G$ is quasi-isometric to any other such metric.  So the group has a canonical quasi-isometry class.  Of course, each Cayley metric is discrete and therefore ``forgets'' the topology on $G$.  However, using a classical theorem of Struble, one may in fact find a right-invariant metric $d$ on $G$ which is not only a representative of the quasi-isometry class of $G$, but in fact generates the topology on $G$.  So there exist right-invariant metrics which convey both the algebraic-topological and the geometric information of $G$.

The groups $\Diff_+^k(M^1)$ we study in this paper are not locally compact, and so other machinery is required.  The primary obstacle in extending the study of large scale geometry to the more general setting of metrizable groups, which may or may not be locally compact, lies in the absence of apparent canonical generating sets\textemdash in particular if $G$ is Polish and not locally compact, then $G$ can never be compactly generated.  Rosendal's solution to this obstacle is the invocation of the relative property (OB), introduced in \cite{rosendal_2009a} and \cite{rosendal_2015a}.  If $G$ is any separable metrizable group, then a subset $\Sigma\subseteq G$ is said to have the \textit{property (OB) relative to $G$} if $\Sigma$ has finite diameter with respect to every right-invariant topologically compatible metric on $G$.  The property (OB) in a separable metrizable group is a suitable analogue for being finite in a countable group or compact in a locally compact group\textemdash indeed, a set in a countable group has property (OB) if and only if it is finite, and a set in a locally compact group has property (OB) if and only if its closure is compact.  $G$ is called \textit{locally (OB)} if it has a neighborhood of identity with the property (OB) relative to $G$.

Rosendal has shown that if $G$ is locally (OB), and is generated as a group by an open subset $\Sigma$ which has the property (OB) relative to $G$, then $G$ admits a topologically compatible right-invariant metric $d$, which assigns finite diameter precisely to those subsets which are relatively (OB) in $G$, and which is quasi-isometric to the right-invariant Cayley metric $d_\Sigma$.  Consequently, all Cayley metrics $d_\Sigma$, where $\Sigma$ is an open set with the property (OB) in $G$, are mutually quasi-isometric, and we are thus able to speak of a canonical quasi-isometry class for $G$.

Some groups, such as the group $\Homeo_+(I)$ of increasing homeomorphisms of the interval, and the infinite permutation group $S_\infty$, are locally (OB) and (OB)-generated but have a trivial quasi-isometry type, i.e. the quasi-isometry type of a singleton, despite being quite large, non-locally compact groups.  Rosendal has identified other non-locally compact groups which have a non-trivial quasi-isometry type, and in many cases he has actually computed this type by displaying an ``understandable'' space as a type representative.  For instance the group $\Aut(\mathcal{T})$ of automorphisms of the $\aleph_0$-regular tree is quasi-isometric to the tree $\mathcal{T}$ itself (see \cite{rosendal_2015a}, \cite{rosendal_2015b} for this and other examples).

\subsection{Main Results}

Our first task in the present work is to address the question: exactly which subsets of $\Diff_+^k(I)$ and $\Diff_+^k(\mathbb{S}^1)$ have the relative property (OB), and thus play the analogue of compact sets in these groups?  We classify the property (OB) in these groups in terms of an easily recognizable boundedness property.

\begin{thm}[see Theorem \ref{ob_character}]  Let $1\leq k\leq\infty$ and let $M^1=I$ or $M^1=\mathbb{S}^1$.  A subset $A\subseteq\Diff_+^k(M^1)$ has the relative property (OB) if and only if $\displaystyle\sup_{f\in A}\sup_{x\in M^1}|\log f'(x)|<\infty$ and $\displaystyle\sup_{f\in A}\sup_{x\in M^1}|f^{(j)}(x)|<\infty$ for every integer $2\leq j\leq k$.
\end{thm}

It seems reasonable to conjecture that an appropriate analogue of the characterization above will hold for relatively (OB) sets in the diffeomorphism group of a general compact manifold of arbitrary dimension.

Our characterization yields the following corollary.

\begin{thm}  $\Diff_+^k(I)$ and $\Diff_+^k(\mathbb{S}^1)$ are locally (OB) and have a well-defined, non-trivial quasi-isometry class, if and only if $k<\infty$.
\end{thm}

So $\Diff_+^\infty(I)$ and $\Diff_+^\infty(\mathbb{S}^1)$ fail to be generated by open relatively (OB) subsets, and therefore do not have a canonical quasi-isometry type in the sense of \cite{rosendal_2015a}.  However, for $1\leq k<\infty$, studying the geometry of these groups is a well-formulated objective.

At the moment we do not know exactly the quasi-isometry type of these groups for $k\geq 2$.  However, for $k=1$ we completely classify the geometry, in the main theorem below.

\begin{thm}[see Corollary \ref{cor_main}]  The following are mutually quasi-isometric:
\begin{enumerate}
		\item $\Diff_+^1(I)$ as a topological group;
		\item $\Diff_+^1(\mathbb{S}^1)$ as a topological group;
		\item $C[0,1]$ as a Banach space;
		\item $C[0,1]$ as an additive topological group.
\end{enumerate}
\end{thm}

So these groups have a very large and complicated large scale geometry.  In particular every separable Banach space embeds quasi-isometrically into both $\Diff_+^1(I)$ and $\Diff_+^1(\mathbb{S}^1)$, since $C[0,1]$ is a universal separable metric space by a classical theorem of Banach and Mazur \cite{banach_mazur_1932a}.  The quasi-isometric equivalence of (3) and (4) above was proved by Rosendal.

Actually we show a little bit more: we display concrete, easily understandable right-invariant topologically compatible metrics on $\Diff_+^1(I)$ and $\Diff_+^1(\mathbb{S}^1)$ which are representatives of each respective group's quasi-isometry class.  Fixing these metrics, we display natural mappings which make $\Diff_+^1(I)$ isometric to a closed subspace of codimension $1$ in $C[0,1]$, and $\Diff_+^1(\mathbb{S}^1)$ quasi-isometric to a closed subspace of codimension $2$ in $C[0,1]$.  Using the well-known fact that $C[0,1]$ is isomorphic to each of its closed hyperplanes, we obtain the result above.\\

\noindent \textit{Acknowledgments.}  I would like to thank C. Rosendal, B. Sari, and A. Akhmedov for helpful discussions and remarks.

\section{Notation and Preliminaries}

\subsection{Coarse Metric Geometry}

Although the theory of coarse geometry of topological groups may be developed in some tremendous generality, we are really only concerned with the groups $\Diff_+^k(M^1)$, which are connected Polish groups.  We present the basic facts we need and refer the reader to \cite{rosendal_2015a} for a detailed study.

If $A$ is a subset of a connected Polish group $G$, then $A$ has the property (OB) relative to $G$ if and only if the following property holds: for every open neighborhood of identity $U\subseteq G$, there exists a positive integer $r$ so that $A\subseteq U^r$.  It follows immediately from this characterization, for instance, that all compact subsets of connected Polish groups have the relative property (OB).  If $A$ and $B$ are relatively $(OB)$ subsets of $G$ then so is $AB$.

Let $G$ be a Polish group and $d$ a metric on $G$.  Then $d$ is called \textit{coarsely proper} if for every $A\subseteq G$, $A$ has the property (OB) relative to $G$ if and only if $\diam_d(A)<\infty$.

The metric space $(G,d)$ is called \textit{large-scale geodesic} if there exists a constant $K\geq 1$ so that for all $f,g\in G$, there are $f=\ell_0,\ell_1,\ell_2,...,\ell_n=g$ in $G$ so that $d(\ell_{i-1},\ell_i)\leq K$ for $1\leq i\leq n$, and $\displaystyle\sum_{i=1}^nd(\ell_{i-1},\ell_i)\leq K\cdot d(f,g)$.  For instance, if $(G,d)$ is a geodesic space in the usual sence, then $(G,d)$ is large-scale geodesic with constant $K=1$.

We need the following central characterization for representatives of the quasi-isometry class of $G$:

\begin{lemma}[\cite{rosendal_2015a} Proposition 47] \label{prop_47}  Let $d$ be a topologically compatible right-invariant metric on a topological group $G$.  The following are equivalent:
\begin{enumerate}
		\item $d$ is coarsely proper and $(G,d)$ is large-scale geodesic;
		\item $d$ is quasi-isometric to each right Cayley metric given by a relatively (OB) symmetric generating set in $G$.
\end{enumerate}
\end{lemma}

\subsection{Banach Spaces}

Beginning now and throughout the remainder of Section 2, we consider integers $0\leq k<\infty$.  For such $k$, we let $C^k[0,1]$ denote the Banach space of real-valued functions of class $C^k$ with domain $[0,1]$.  We also denote $C[0,1]=C^0[0,1]$.  For any such space $X=C^k[0,1]$, we set\\

\begin{center} $X^\#=\{f\in X:f(0)=0\}$.
\end{center}
\vspace{.3cm}

For any space $X$ of the form $X=C^k[0,1]$ or $X=C^k[0,1]^\#$ or $X=C^k[0,1]\oplus\mathbb{R}^n$ ($n\geq 1$), we let $\|\cdot\|$ denote the supremum norm on $X$.  So $\|\cdot\|$ induces the usual Banach space structure on $X=C[0,1]$ or $X=C[0,1]^\#$ or $X=C[0,1]\oplus\mathbb{R}^n$, but not on $X=C^k[0,1]$.  The usual norm on $C^k[0,1]$ will be denoted by $\|\cdot\|_k$.  (So $\|f\|_k=\displaystyle\sum_{i=0}^k\|f^{(i)}\|$ for each $f\in C^k[0,1]$.)

We permanently fix the antiderivative mapping $\mathcal{I}:C[0,1]\oplus\mathbb{R}\rightarrow C^1[0,1]$ defined by\\

\begin{center} $\mathcal{I}(f,b)(x)=\int_0^x f(t)dt+b$ for all $x\in[0,1]$.
\end{center}
\vspace{.3cm}

Note that for each $k\geq 0$, $\|\mathcal{I}(f,b)\|_{k+1}\leq\|f\|_k+\|b\|\leq 2\max(\|f\|_k,\|b\|)$.  So the restriction of $\mathcal{I}$ to $C^k[0,1]\oplus\mathbb{R}$ is a bounded linear operator from $C^k[0,1]\times\mathbb{R}$ into $C^{k+1}[0,1]$ with $\|\mathcal{I}\|=2$.

We also define $\iota:C[0,1]\rightarrow C^1[0,1]$ by $\iota(f)=\mathcal{I}(f,0)$.  The following lemma is straightforward to check.

\begin{lemma} \label{lemma_iota}  For each $k\geq0$, the restriction of the mapping $\mathcal{I}$ to $C^k[0,1]\oplus\mathbb{R}$ is a Banach space isomorphism (and hence a homeomorphism) of $C^k[0,1]\oplus\mathbb{R}$ onto $C^{k+1}[0,1]$.  
The restriction of $\iota$ to $C^k[0,1]$ is a Banach space isomorphism (and hence a homeomorphism) of $C^k[0,1]$ onto $C^{k+1}[0,1]^\#$.
\end{lemma} \hfill $\qed$

\subsection{Interval and Circle Diffeomorphisms}

As we study the groups $\Diff_+^k(\mathbb{S}^1)$, we wish to avoid discussion of charts and universal coverings, so we choose the simplest available conventions.  We imagine $\mathbb{S}^1$ as the topological quotient space $I/E$, where $E$ is the equivalence relation on $I$ identifying $0$ and $1$.  We define the distance $d_{\mathbb{S}^1}$ in the circle via a natural parametrization, say, $d_{\mathbb{S}^1}([x],[y])=\|e^{2\pi ix}-e^{2\pi iy}\|$ for all (equivalence classes) $[x],[y]\in\mathbb{S}^1=I/E$, where the norm is taken in $\mathbb{C}$.  Whenever $x\in\mathbb{R}$ and $0\leq x\leq 1$, we adopt the habit of identifying $x$ with its equivalence class $[x]$, and understanding $x\in\mathbb{S}^1$.

To each homeomorphism $f:\mathbb{S}^1\rightarrow\mathbb{S}^1$ we associate its unique lift map $\tilde{f}:\mathbb{R}\rightarrow\mathbb{R}$, which satisfies: (1) $\tilde{f}(0)\in[0,1)$; (2) $\tilde{f}(x)+1=\tilde{f}(x+1)$ for all $x\in\mathbb{R}$; and (3) $f(x \mod 1)=\tilde{f}(x) \mod 1$ for all $x\in\mathbb{R}$.  We say such a homeomorphism $f$ is of class $C^k$ if and only if $\tilde{f}$ is of class $C^k$ (which necessarily implies $\tilde{f}^{(j)}(0)=\tilde{f}^{(j)}(1)$ for all $0\leq j\leq k$).  In general we define the derivatives $f^{(j)}(x)=\tilde{f}^{(j)}(x)$ for each $x\in\mathbb{S}^1$, each $f\in\Diff_+^k(\mathbb{S}^1)$, and each $1\leq j\leq k$.  It is straightforward to verify that these derivatives satisfy the usual chain rule, product rule, quotient rule, and inversion rule.

By a small abuse of notation, we regard the group of rotations $\mathbb{S}^1$ as a closed subgroup of $\Diff_+^k(\mathbb{S}^1)$, by identifying each $[t]\in\mathbb{S}^1$ with the mapping $[x]\mapsto[(x+t)\mod 1]$.

For the remainder of this section let either $M^1=I$ or $M^1=\mathbb{S}^1$.  A standard compatible metric $\rho_k$ on $\Diff_+^k(M^1)$ is given by the following:

\begin{center} $\rho_k(f,g)=\displaystyle\sup_{x\in M^1}d_{M^1}(f(x),g(x))+\sum_{j=1}^k\|f^{(j)}-g^{(j)}\|$.
\end{center}
\vspace{.3cm}

We introduce some more permanent notation.  Define a map $\phi_1:\Diff_+^1(M^1)\rightarrow C[0,1]^\#$ by\\

\begin{center} $\phi_1(f)=\log f'-\log f'(0)$.
\end{center}
\vspace{.3cm}

Note that for each $k\geq 2$, the restriction of the mapping $\phi_1$ to the subgroup $\Diff_+^k(M^1)$ takes values in $C^{k-1}[0,1]^\#$.  For each $k\geq 2$ define a map $\phi_k:\Diff_+^k(M^1)\rightarrow C[0,1]$ by\\

\begin{center} $\phi_k(f)=(\phi_1(f))^{(k-1)}$.
\end{center}
\vspace{.3cm}

So we have defined $\phi_2$ to be the logarithmic derivative $\phi_2(f)=\frac{f''}{f'}$, and each $\phi_k(f)$ is just $(\phi_{k-1}(f))'$.  The maps $\phi_k$ serve as ``alternative derivatives'' for us that are better behaved than the usual ones.  Each map $\phi_k$ is continuous.

The following lemmas are elementary calculus that we will need later.

\begin{lemma} \label{lemma_mandrill}  For each $k\geq 2$, there exists a polynomial $Q_k$ in $2(k-1)$ variables so that\\

\begin{center} $f^{(k)}=Q_k(f',f'',...,f^{(k-1)},\phi_2(f),\phi_3(f),...,\phi_k(f))$
\end{center}
\vspace{.3cm}

\noindent for every $f\in \Diff_+^k(M^1)$.
\end{lemma}

\begin{proof}  By induction on $k$.  For the base case, note that $f''=f'\phi_2(f)$, so setting $Q_1(x,y)=xy$, we are done.  Now the inductive case is just applications of the product rule.
\end{proof}

\begin{lemma} \label{lemma_echidna}  For each $k\geq 1$, there exists a polynomial $R_k$ in $k$ variables so that\\

\begin{center} $(g^{-1})^{(k)}=\left[\dfrac{R_k(g',g'',...,g^{(k)})}{(g')^{2k-1}}\right]\circ g^{-1}$
\end{center}
\vspace{.3cm}

\noindent for every $g\in\Diff^k_+(M^1)$.
\end{lemma}

\begin{proof}  By induction on $k$.  Taking $R_1(x)=1$, we get the base case by the inversion rule $(g^{-1})'=\dfrac{1}{g'}\circ g^{-1}$.  The inductive case now follows from the product rule, quotient rule, and chain rule.
\end{proof}

\begin{rem}  The polynomials $Q_k$ in Lemma \ref{lemma_mandrill} may be computed using the Leibniz rule for $k$-th order derivatives of products.  The polynomials $R_k$ in Lemma \ref{lemma_echidna} are more difficult to compute, but it is also possible to derive them by using the classical Fa\'{a} di Bruno's formula for $k$-th order derivatives of compositions of functions.  We omit the details-- all we really care about is that $Q_k$ and $R_k$ are polynomials, thus continuous and bounded on compact sets.  We will use these facts in the proofs of Lemma \ref{lemma_muledeer} and Corollary \ref{cor_radishes}.
\end{rem}

\section{Characterizing the Relative Property (OB) in $\Diff_+^k(M^1)$}

The goal of this section is to show that a subset of $\Diff^k_+(M^1)$ has the relative property (OB) if and only if it has uniformly log-bounded first derivatives, and uniformly bounded higher order derivatives for each order greater than $1$.

The standard metrics $\rho_k$ on $\Diff_+^k(M^1)$ are cumbersome for establishing this fact.  So our strategy is to associate with each group $\Diff_+^k(M^1)$ a natural Banach space consisting of ``initial value problems,'' via certain continuous mappings $\Phi_k$ (see Figure 1 below).  These maps are homeomorphisms when $M^1=I$, but not injective nor surjective when $M^1=\mathbb{S}^1$.  We use the maps $\Phi_k$ to pull back the norm and define a corresponding continuous pseudometric on $\Diff_+^k(M^1)$, which we denote by $d_k$.  We have $d_k$ is a metric if $M^1=I$ but not if $M^1=\mathbb{S}^1$.  These pseudometrics $d_k$ turn out to be well suited for our computations.

\begin{center}
\begin{figure}
\begin{tabular}{ccc}
$\Diff_+^1(M^1)$ & $\xrightarrow{\Phi_1}$ & $C[0,1]^\#$\\ [2ex]
\rotatebox[origin=c]{90}{$\subseteq$} & & $\big\uparrow$\rlap{\tiny $\iota$}\\ [2ex]
$\Diff_+^2(M^1)$ & $\xrightarrow{\Phi_2}$ & $C[0,1]$\\ [2ex]
\rotatebox[origin=c]{90}{$\subseteq$} & & $\big\uparrow$\rlap{\tiny $\mathcal{I}$}\\ [2ex]
$\Diff_+^3(M^1)$ & $\xrightarrow{\Phi_3}$ & $C[0,1]\oplus\mathbb{R}$\\ [2ex]
\rotatebox[origin=c]{90}{$\subseteq$} & & $\big\uparrow$\rlap{\tiny $\mathcal{I}\oplus\id$}\\ [2ex]
$\Diff_+^4(M^1)$ & $\xrightarrow{\Phi_4}$ & $C[0,1]\oplus\mathbb{R}^2$\\ [2ex]
\rotatebox[origin=c]{90}{$\subseteq$} & & $\big\uparrow$\rlap{\tiny $\mathcal{I}\oplus\id\oplus\id$}\\ [2ex]
$\vdots$ & & $\vdots$
\end{tabular}
\vspace{.3cm}
\caption{Road map for studying (pseudo-) metrics on $\Diff_+^k(M^1)$.  All upward arrows in the right column are continuous injective linear maps which are never surjective, but which are Banach space isomorphisms onto their range.  The diagram commutes.}
\end{figure}
\end{center}
\vspace{.3cm}



We begin by setting $\Phi_1=\phi_1$ and $\Phi_2=\phi_2$.  For all $k\geq 3$, define maps $\Phi_k:\Diff_+^k(I)\rightarrow C[0,1]\times\mathbb{R}^{k-2}$ by\\

\begin{center} $\Phi_k(f)=(\phi_k(f),~\phi_{k-1}(f)(0),~\phi_{k-2}(f)(0),~...,~\phi_2(f)(0))$.
\end{center}
\vspace{.3cm}

So, for instance, $\Phi_4(f)=(\phi_4(f), \phi_3(f)(0), \phi_2(f)(0))$ for any $f\in\Diff_+^4(I)$.  We note that each $\Phi_k$ sends the identity element $e\in\Diff_+^k(M^1)$ to the zero vector.

It is clear that each map $\Phi_k$ is continuous.  So we define a continuous pseudometric $d_k$ on $\Diff_+^k(M^1)$ by the rule\\

\begin{center} $d_k(f,g)=\|\Phi_k(f)-\Phi_k(g)\|$.
\end{center}
\vspace{.3cm}

So every open $d_k$-ball in $\Diff_+^k(M^1)$ is an open set in $\Diff_+^k(M^1)$.  In case $M^1=I$, we can say more: by Lemma \ref{lemma_homeomorphism} below, $d_k$ is a compatible metric on $\Diff_+^k(I)$, which renders it isometric with the target Banach space of $\Phi_k$.

\begin{lemma} \label{lemma_homeomorphism} 
\begin{enumerate}
		\item The mapping $\Phi_1$ is a homeomorphism of $\Diff_+^k(I)$ onto $C^{k-1}[0,1]^\#$.
		\item For each $k\geq 2$, the mapping $\Phi_2$ is a homeomorphism of $\Diff_+^k(I)$ onto $C^{k-2}[0,1]$.  
		\item For each $k\geq 2$, the mapping $\Phi_k$ is a homeomorphism of $\Diff_+^k(I)$ onto $C[0,1]\times\mathbb{R}^{k-2}$.
\end{enumerate}
\end{lemma}

\begin{proof}   It is easy to see that $\Phi_1$ is injective.  If $F\in C^{k-1}[0,1]^\#$ for $k\geq 1$, then one may set $C=\int_0^1\exp(F(t))dt$, and observe that $f=\Phi_1^{-1}(F)$ is given by the formula\\

\begin{center} $f(x)=\Phi_1^{-1}(F)(x)=\frac{1}{C}\int_0^x\exp(F(t))dt$ for $x\in[0,1]$.
\end{center}
\vspace{.3cm}

This map $f$ is increasing on $I$, with positive derivative $\frac{1}{C}\exp\circ F$ of class $C^{k-1}$, so $f$ is of class $C^k$.  It also fixes $0$ and $1$, so $f\in\Diff_+^k(I)$ and therefore the restriction mapping $\Phi_1|_{\Diff_+^k(I)}$ is in fact a bijection onto $C^{k-1}[0,1]^\#$.  It is also clear that $\Phi_1^{-1}|_{C^{k-1}[0,1]^\#}$ as we have displayed it is continuous with respect to the uniform $C^k$-topology on $\Diff_+^k(I)$ and the uniform $C^{k-1}$-topology on $C^{k-1}[0,1]^\#$.  So $\Phi_1$ is a homeomorphism of $\Diff_+^k(I)$ onto $C^{k-1}[0,1]^\#$ for each $k$ as claimed.

Also, note that $\Phi_2$ is exactly the composition mapping $\Phi_2=\iota^{-1}\circ\Phi_1$.  Since $\iota^{-1}$ is a homeomorphism from $C^1[0,1]^\#$ onto $C[0,1]$, we see that $\Phi_2$ is a composition of homeomorphisms and hence a homeomorphism.

Similarly, for each $k\geq 3$, the mapping $\Phi_k$ is exactly the composition mapping\\

\begin{center} $\Phi_k=(\mathcal{I}\oplus\underbrace{\id\oplus\dots\oplus\id}_{k-3})^{-1}\circ\dots\circ(\mathcal{I}\oplus\id)^{-1}\circ\mathcal{I}^{-1}\circ\iota^{-1}\circ \Phi_1$.
\end{center}
\vspace{.3cm}

\noindent So each $\Phi_k$ is a composition of homeomorphisms and hence a homeomorphism.


\end{proof}

\begin{lemma} \label{lemma_muskox}  Let $f\in\Diff_+^k(M^1)$ and $\delta>0$.  If $d_1(f,e)<\delta$, then $e^{-2\delta}\leq f'\leq e^{2\delta}$ and $e^{-2\delta}\leq (f^{-1})'\leq e^{2\delta}$.
\end{lemma}

\begin{proof}  Let $F=\Phi_1(f)\in C[0,1]^\#$, so by hypothesis $\|F\|<\delta$.  We gave a formula for $\Phi_1^{-1}$ in the proof of Lemma \ref{lemma_homeomorphism}, where we regard $\Phi_1^{-1}$ as a mapping from $C[0,1]^\#$ to $\Diff_+^1(I)$.  If $M^1=I$, then $f=\Phi_1^{-1}(F)$, and if $M^1=\mathbb{S}^1$, then $f'([x])$ agrees with $(\Phi_1^{-1}(F))'(x)$ at each $x\in I$.

By taking the derivative of the formula for $\Phi_1^{-1}$, we compute that\\

\begin{center} $f'(0)=\dfrac{1}{\int_0^1\exp(F(t))dt}$.
\end{center}
\vspace{.3cm}

Now $\|F\|<\delta$ implies that $e^{-\delta}<\|\exp(F)\|<e^\delta$, and therefore $e^{-\delta}\leq\int_0^1\exp(F(t))dt\leq e^\delta$.  So reciprocating, we get $e^{-\delta}\leq f'(0)\leq e^\delta$.  By hypothesis we also have $e^{-\delta}<\frac{f'}{f'(0)}<e^\delta$.  Combining the two inequalities, we get $e^{-2\delta}\leq f'\leq e^{-2\delta}$.  The same inequality follows immediately for $(f^{-1})'$ by the inverse rule for derivatives.
\end{proof}

\begin{lemma} \label{lemma_garbanzo}  For each $k\geq 2$ and $j<k$, the identity mapping $\id:(\Diff_+^k(M^1),d_k)\rightarrow(\Diff_+^k(M^1),d_j)$ is $2^{k-j}$-Lipschitz at the identity element $e$.
\end{lemma}

\begin{proof}  We show the result for $j=k-1$ by induction on $k$, whence the result for general $j<k$ follows immediately.

For $k=2$, we note that for every $f\in\Diff_+^2(M^1)$, $d_1(f,e)=\|\phi_1(f)\|\leq\|\phi_1(f)\|_1=\|\iota(\phi_2(f))\|_1$.  Now $\|\iota\|\leq 2$ by the remarks preceding Lemma \ref{lemma_iota}, and therefore $d_1(f,e)\leq2\|\phi_2(f)\|=2d_2(f,e)$, which proves the base case.  Similarly, for the inductive case, we note that if $f\in\Diff_+^k(M^1)$ with $k\geq3$, we have 

\begin{align*}
d_{k-1}(f,e) &= \|(\phi_{k-1}(f),\phi_{k-2}(f)(0),...,\phi_2(f)(0))\|\\
&= \max\{\|\phi_{k-1}(f)\|,\|(\phi_{k-2}(f)(0),...,\phi_2(f)(0))\|\}\\
&\leq \max\{\|\phi_{k-1}(f)\|_1,\|(\phi_{k-2}(f)(0),...,\phi_2(f)(0))\|\}\\
&\leq \max\{\|\mathcal{I}(\phi_k(f),\phi_{k-1}(f)(0))\|_1,\|\Phi_k(f)\|\}\\
&\leq \max\{\|\mathcal{I}\|\|(\phi_k(f),\phi_{k-1}(f)(0)\|,\|\Phi_k(f)\|\}\\
&\leq \max\{2\|\Phi_k(f)\|,\|\Phi_k(f)]\|\}\\
&= 2\|\Phi_k(f)\|\\
&= 2d_k(f,e).
\end{align*}
\end{proof}

\begin{lemma} \label{lemma_muledeer}  For each $k\geq 1$, the identity mapping $\id:(\Diff_+^k(M^1),d_k)\rightarrow(\Diff_+^k(M^1),\rho_k)$ takes $d_k$-bounded sets to $\rho_k$-bounded sets.
\end{lemma}

\begin{proof}  It suffices to show that for every $\delta>0$, there is $\epsilon>0$ so that $B_{d_k}(e,\delta)\subseteq B_{\rho_k}(e,\epsilon)$.

We proceed by induction on $k$.  First consider $k=1$ and let $\delta>0$ be given.  Let $\epsilon>2+e^{2\delta}-1$.  If $f\in\Diff_+^1(M^1)$ satisfies $d_1(f,e)<\delta$, then $||\phi_1(f)||<\delta$, whence $e^{-2\delta}<f'<e^{2\delta}$ by Lemma \ref{lemma_muskox}.  So $e^{-2\delta}-1<f'-1<e^{2\delta}-1$, whence $\rho_1(f,e)=\displaystyle\sup_{x\in M^1}d_{M^1}(f(x),x)+\|f'-1\|\leq \diam_{d_M^1}(M^1)+e^{2\delta}-1\leq 2+e^{2\delta}-1<\epsilon$.  Therefore $E\subseteq B_{\rho_1}(e,\epsilon)$, proving the base case.

Now fix some $k\geq 2$, and inductively suppose that every $d_j$-bounded set is also $\rho_j$-bounded, for every $j<k$.  Let $\delta>0$ be given.  Then $B_{d_k}(e,\delta)\subseteq B_{d_j}(e,2^{k-j}\delta)$ by the Lipschitz property of the identity mapping with respect to these metrics (Lemma \ref{lemma_garbanzo}).  So our inductive hypothesis says there is a $\gamma>2^k\delta$ so that $B_{d_k}(e,\delta)\subseteq B_{\rho_j}(e,\gamma)$, for every $j<k$.

Let $Q_k$ be the polynomial in $2(k-1)$ variables guaranteed by Lemma \ref{lemma_mandrill}, and let $\epsilon>0$ be large enough that $\epsilon>2\gamma$ and

\begin{center} $\epsilon>2\cdot\displaystyle\sup_{\vec{x}\in[-\gamma,\gamma]^{2(k-1)}}|Q_k(\vec{x})|$.
\end{center}
\vspace{.3cm}

If $f\in\Diff_+^k(M^1)$ satisfies $d_k(f,e)<\delta$, then $d_j(f,e)<2^k\delta<\gamma$ for all $1\leq j\leq k$, and $\rho_j(f,e)<\gamma$ for all $1\leq j< k$ by the inductive hypothesis.  In particular $\|f^{(j)}\|<\gamma$ for all $1\leq j<k$.  It follows then from Lemma \ref{lemma_mandrill} and our choice of $\epsilon$ that 

\begin{align*}
\|f^{(k)}\| &= \displaystyle\sup_{x\in M^1}|f^{(k)}(x)|\\
&= \displaystyle\sup_{x\in M^1}|Q_k(f,f',...,f^{(k-1)},\phi_2(f),...,\phi_k(f))|(x)\\
&= \displaystyle\sup_{x\in M^1}|Q_k(f(x),f'(x),...,f^{(k-1)}(x),\phi_2(f)(x),...,\phi_k(f)(x))|\\
&\leq \displaystyle\sup_{x\in M^1}\epsilon/2=\epsilon/2.
\end{align*}
\vspace{.3cm}


Also $\rho_{k-1}(f,e)<\gamma<\epsilon/2$, so $\rho_k(f,e)=\rho_{k-1}(f,e)+\|f^{(k)}\|<\epsilon$.  This concludes the inductive step and the proof.
\end{proof}

\begin{cor} \label{cor_radishes}  For each $k\geq 1$, the inverse mapping $h\mapsto h^{-1}$, $(\Diff_+^k(M^1),d_k)\rightarrow(\Diff_+^k(M^1),\rho_k)$ takes $d_k$-bounded sets to $\rho_k$-bounded sets.
\end{cor}

\begin{proof}  As in the previous proof, it suffices to show that for every $\delta>0$, there is $\epsilon>0$ so that $(B_{d_k}(e,\delta))^{-1}\subseteq B_{\rho_k}(e,\epsilon)$.

Let $k\geq 1$ and let $\delta>0$ be given.  By the previous Lemma \ref{lemma_muledeer}, there is $\gamma>0$ so that $B_{d_k}(e,\delta)\subseteq B_{\rho_j}(e,\gamma)$ for all $j\leq k$.  For each $j\leq k$, let $R_j$ be the polynomial in $j$ variables guaranteed by Lemma \ref{lemma_echidna}.  Let $\epsilon>0$ be so large that $\epsilon/(k+1)>\max\{e^{2^k\delta},\diam_{d_{M^1}}(M^1)\}$, and\\

\begin{center} $\epsilon/(k+1)>\displaystyle\max_{1\leq j\leq k}\sup_{\vec{x}\in[-(\gamma+1),\gamma+1]^j}\dfrac{|R_j(\vec{x})|}{(e^{-2^k\delta})^{2j-1}}$.
\end{center}
\vspace{.3cm}

Now if $h\in\Diff_+^k(M^1)$ satisfies $d_k(h,e)<\delta$, then $\rho_k(h,e)<\gamma$ and therefore in particular $\|h'\|,\|h''\|,...,\|h^{(k)}\|<\gamma+1$.  By the Lipschitz property of the mapping $\id:(\Diff_+^k(M^1),d_k)\rightarrow(\Diff_+^k(M^1),d_1)$ (Lemma \ref{lemma_garbanzo}), we also have $d_1(h,e)<2^{k-1}\delta$, and thus by Lemma \ref{lemma_muskox} we get $e^{-2^k\delta}<(h^{-1})'<e^{2^k\delta}$.  Thus $\|(h^{-1})'-1\|<e^{2^k\delta}<\epsilon/(k+1)$.  Now it also follows from Lemma \ref{lemma_echidna} and our choice of $\epsilon$ that 

\begin{align*}
\|(h^{-1})^{(j)}\| &= \displaystyle\sup_{x\in M^1}|(h^{-1})^{(j)}(x)|\\
&= \displaystyle\sup_{x\in M^1}\left|\dfrac{R_j(h',h'',...,h^{(j)})}{(h')^{2j-1}}\circ h^{-1}\right|(x)\\
&= \displaystyle\sup_{x\in M^1}\left|\dfrac{R_j(h',h'',...,h^{(j)})}{(h')^{2j-1}}\right|(x)\\
&= \displaystyle\sup_{x\in M^1}\dfrac{|R_j(h'(x),h''(x),...,h^{(j)}(x))|}{(e^{-2^k\delta})^{2j-1}}\\
&\leq \displaystyle\sup_{x\in M^1}\epsilon/(k+1) = \epsilon/(k+1)
\end{align*}
\vspace{.3cm}

\noindent for all $2\leq j\leq k$.  So we have $\rho_k(h^{-1},e)=\displaystyle\sup_{x\in M^1}d_{M^1}(h^{-1}(x),x)+\|(h^{-1})'-1\|+\displaystyle\sum_{j=2}^k\|(h^{-1})^{(j)}\|<\epsilon/(k+1)+\epsilon/(k+1)+(k-1)\cdot\epsilon/(k+1)=\epsilon$.  Since $h$ was arbitrary, we have shown $(B_{d_k}(e,\delta))^{-1}\subseteq B_{\rho_k}(e,\epsilon)$ as claimed.
\end{proof}






\begin{lemma} \label{lemma_wombat}  For all $f,h\in\Diff_+^k(M^1)$, $\phi_1(fh^{-1})=[\phi_1(f)-\phi_1(h)]\circ h^{-1}$.  For each $k\geq 2$, there exist polynomials $P^k_2,P^k_3,...,P^k_k$ in $1, 2,..., k-1$ variables respectively, with the property that for all $f,h\in\Diff_+^k(M^1)$,

\begin{align*}
\phi_k(fh^{-1}) = & [\phi_k(f)-\phi_k(h)]\circ h^{-1}\cdot P^k_2((h^{-1})')\\
&+ [\phi_{k-1}(f)-\phi_{k-1}(h)]\circ h^{-1}\cdot P^k_3((h^{-1})',(h^{-1})'')\\
&+ ...\\
&+ [\phi_3(f)-\phi_3(h)]\circ h^{-1}\cdot P^k_{k-1}((h^{-1})',(h^{-1})'',...,(h^{-1})^{(k-2)})\\
&+ [\phi_2(f)-\phi_2(h)]\circ h^{-1}\cdot P^k_k((h^{-1})',(h^{-1})'',...,(h^{-1})^{(k-1)}).
\end{align*}
\end{lemma}

\begin{proof}  The fact that $\phi_1(fh^{-1})=[\phi_1(f)-\phi_1(h)]\circ h^{-1}$ is an easy computation.  The longer claim is established by induction on $k$.  By the chain rule and the linearity of the derivative,\\

\begin{align*}
\phi_2(fh^{-1}) &= ([\phi_1(f)-\phi_1(h)]\circ h^{-1})'\\
&= [\phi_2(f)-\phi_2(h)]\circ h^{-1}\cdot (h^{-1})'.
\end{align*}
\vspace{.3cm}

Taking $P^2_2(x)=x$, this establishes the base case $k=2$.  The inductive case is once again a straightforward consequence of the chain rule, the product rule, and the linearity of the derivative.
\end{proof}







\begin{lemma} \label{lemma_main}  The pseudometric $d_1$ is right-invariant.  For each $k\geq 2$, there exists a non-decreasing function $A_k:(0,\infty)\rightarrow(0,\infty)$ with the property that for all $f,g,h\in\Diff_+^k(M^1)$,\\

\begin{center} $d_k(fh^{-1},gh^{-1})\leq A_k(\|\Phi_k(h)\|)\cdot d_k(f,g)$.
\end{center}
\vspace{.3cm}

In other words, right translation by $h^{-1}$ is Lipschitz with a constant depending only on the magnitude of $\Phi_k(h)$.
\end{lemma}

\begin{proof}  By Lemma \ref{lemma_wombat}, for each $f,h\in\Diff_+^1(M^1)$ we have $d_1(f,h)=\|\phi_1(f)-\phi_1(h)\|=\|[\phi_1(fh^{-1})]\circ h^{-1}\|=\|\phi_1(fh^{-1})\|=\|0-\phi_1(fh^{-1})\|=d_1(e,fh^{-1})$.  It follows from this observation that for every $f,g,h\in\Diff_+^1(I)$, $d_1(fh,gh)=d_1(e,(fh)(hg)^{-1})=d_1(e,fg^{-1})=d_1(f,g)$ and thus $d_1$ is right-invariant as claimed.

Now suppose $k\geq 2$.  Using the fact that inversion takes $d_k$-bounded sets to $\rho_k$-bounded sets (Corollary \ref{cor_radishes}), we see that to each $\delta\in(0,\infty)$, one may assign a constant $\gamma_\delta$ for which $d_k(h,e)<\delta$ implies $\rho_k(h^{-1},e)<\gamma_\delta$.  Without loss of generality we may further assume that the mapping $\delta\mapsto\gamma_\delta$ is a non-decreasing function of $\delta$.  Now let $P^k_2,...,P^k_k$ be the polynomials guaranteed by Lemma \ref{lemma_wombat}.  Define $A_k$ as follows:

\begin{center} $A_k(\delta)=\displaystyle\sup_{x_2\in[-\gamma_\delta,\gamma_\delta]}|P^k_2(x_2)|+\displaystyle\sup_{\vec{x}_3\in[-\gamma_\delta,\gamma_\delta]^2}|P^k_3(\vec{x}_3)|+...+\displaystyle\sup_{\vec{x}_k\in[-\gamma_\delta,\gamma_\delta]^{k-1}}|P^k_k(\vec{x}_k)|$.
\end{center}
\vspace{.3cm}


It is clear from the definition that $A_k$ is non-decreasing, given that $\gamma_\delta$ is non-decreasing.  Now let $f,g,h\in\Diff_+^k(M^1)$ and fix $\delta=\|\Phi_k(h)\|=d_k(h,e)$.  By construction, $\rho_k(h,e)\leq \gamma_\delta$, and in particular $\|(h^{-1})'\|,\|(h^{-1})''\|,...,\|(h^{-1})^{(k)}\|\leq\gamma_\delta$.  Now it follows from Lemma \ref{lemma_wombat} that \\

\begin{align*}
\|\phi_k(fh^{-1})-\phi_k(gh^{-1})\| &= \left\|\displaystyle\sum_{i=1}^{k-1}[\phi_{k-i+1}(f)-\phi_{k-i+1}(h)]\circ h^{-1}\cdot P^k_{i+1}((h^{-1})',...,(h^{-1})^{(i)})\right.\\
& \text{\indent} \left.-\sum_{i=1}^{k-1}[\phi_{k-i+1}(g)-\phi_{k-i+1}(h)]\circ h^{-1}\cdot P^k_{i+1}((h^{-1})',...,(h^{-1})^{(i)})\right\|\\
&= \left\|\sum_{i=1}^{k-1}[\phi_{k-i+1}(f)-\phi_{k-i+1}(g)]\circ h^{-1}\cdot P^k_{i+1}((h^{-1})',...,(h^{-1})^{(i)})\right\|\\
&\leq \sum_{i=1}^{k-1}\|[\phi_{k-i+1}(f)-\phi_{k-i+1}(g)]\circ h^{-1}\|\cdot \sup_{\vec{x}_{i+1}\in[-\gamma_\delta,\gamma_\delta]^i}|P^k_{i+1}(\vec{x}_{i+1})|\\
&\leq \sum_{i=1}^{k-1} d_{k-i+1}(f,g)\cdot \sup_{\vec{x}_{i+1}\in[-\gamma_\delta,\gamma_\delta]^i}|P^k_{i+1}(\vec{x}_{i+1})|\\
&\leq d_k(f,g)\cdot A_k(\delta).
\end{align*}
\vspace{.3cm}

Now the hard work is over, and we can finish the proof with a quick induction on $k$.  For the base case $k=2$, if $f,g,h\in\Diff_+^2(M^1)$, we have $d_2(fh^{-1},gh^{-1})=\|\phi_2(fh^{-1})-\phi_2(gh^{-1})\|\leq A_2(\|\Phi_2(h)\|)d_2(f,g)$ by the above, and we are done.  For the inductive case, fix $k\geq 3$ and $h\in\Diff_+^k(I)$, and assume the mapping $f\mapsto fh^{-1}$, $(\Diff_+^j(I),d_j)\rightarrow(\Diff_+^j(I),d_j)$ is $A_j(\|\Phi_j(h)\|)$-Lipschitz for each $j<k$.  We may assume that $A_k\geq A_j$ for all $j<k$, by replacing $A_k$ with $\displaystyle\max_{1\leq j\leq k}A_j$ if necessary.  Then for any $f,g\in\Diff_+^k(M^1)$, by the inductive hypothesis and the computation above, we have\\

\begin{align*}
d_k(fh^{-1},gh^{-1}) &= \|\Phi_k(fh^{-1})-\Phi_k(gh^{-1})\|\\
&= \max\{\phi_k(fh^{-1})-\phi_k(gh^{-1}),[\phi_{k-1}(fh^{-1})-\phi_{k-1}(gh^{-1})](0),\\
& \dots,[\phi_2(fh^{-1})-\phi_2(gh^{-1})](0)\}\\
&\leq \max\{A_k(\|\Phi_k(h)\|)\cdot d_k(f,g),A_{k-1}(\|\Phi_{k-1}\|)\cdot d_{k-1}f,g),\\
& \dots,A_2(\|\Phi_2(h)\|)\cdot d_2(f,g)\}\\
&\leq A_k(\|\Phi_k(h)\|)\cdot d_k(f,g).
\end{align*}
\vspace{.3cm}

This demonstrates the inductive step.
\end{proof}

\begin{thm}\label{ob_character}  Let $1\leq k\leq\infty$.  A subset $A\subseteq\Diff_+^k(M^1)$ has the relative property (OB) if and only if $\displaystyle\sup_{f\in A}\sup_{x\in M^1}|\log f'(x)|<\infty$ and $\displaystyle\sup_{f\in A}\sup_{x\in M^1}|f^{(j)}(x)|<\infty$ for every integer $2\leq j\leq k$.
\end{thm}

\begin{proof}  First fix any $1\leq k\leq \infty$ and suppose $A$ has the property (OB) relative to $\Diff_+^k(M^1)$.  Let $U$ be the $d_1$-ball of radius $1$ about identity, so $U$ is open in $\Diff_+^k(M^1)$.  Since $A$ has the relative property (OB) and $\Diff_+^k(M^1)$ is connected, there is a positive integer $r$ so that $A\subseteq U^r$.  Since $U$ is $d_1$-bounded, and $d_1$ is right-invariant, it follows that $U^r$ is $d_1$-bounded and hence $A$ is $d_1$-bounded.  In particular $\displaystyle\sup_{f\in A}\sup_{x\in M^1}|\log f'(x)-\log f'(0)|<\infty$ and hence $\displaystyle\sup_{f\in A}\sup_{x\in M^1}|\log f'(x)|<\infty$.

Next fix any integer $j$ so that $2\leq j\leq k$, let $V$ be the $d_j$-ball of radius $1$ about identity, and let $U=V\cap V^{-1}$.  Again since $A$ has the relative property (OB) and $U$ is open, there is a positive integer $r$ so that $A\subseteq U^r$.  

Let $N=\displaystyle\sup_{h\in U}\|\Phi_j(h)\|$.  This supremum exists because $U$ is $d_j$-bounded.  Let $C=A_j(N)$, where $A_j$ is the function guaranteed in Lemma \ref{lemma_main}.  Then right translation by $h$ is $C$-Lipschitz for any choice of $h\in U$.  For any $a\in A$, we have $a=u_1u_2...u_r$ for some $u_1,...,u_r\in U=U^{-1}$.  Then by the triangle inequality and repeated applications of the Lipschitz property, we get

\begin{align*}
d_j(a,e) &\leq d_j(u_r,e)+d_j(u_{r-1}u_r,u_r)+...+d_j(u_1u_2...u_r,u_2...u_r)\\
&\leq d_j(u_r,e)+Cd_j(u_{r-1},e)+...+C^{r-1}d_j(u_1,e)\\
&\leq (1+C+...+C^{r-1})N.
\end{align*}
\vspace{.3cm}

So $A$ is $d_j$-bounded.  Therefore $A$ is $\rho_j$-bounded by Lemma \ref{lemma_muledeer}.  It immediately follows that $\displaystyle\sup_{f\in A}\sup_{x\in M^1}|f^{(j)}(x)|<\infty$.  Thus one direction of the theorem is proved.

For the other direction, we consider separately the cases $M^1=I$ and $M^1=\mathbb{S}^1$.  First take $M^1=I$.  Fix $1\leq k\leq\infty$ and suppose $\displaystyle\sup_{f\in A}\sup_{x\in I}|\log f'(x)|<\infty$, and $\displaystyle\sup_{f\in A}\sup_{x\in I}|f^{(j)}(x)|<\infty$ for every $j\leq k$.  Deduce that $A\subseteq B_{d_k}(e,M)$ for some fixed real number $M>0$.  Let $U\subseteq\Diff_+^k(I)$ be an arbitrary basic open set, so $U=B_{d_j}(e,\epsilon)$ for some $j\leq k$ and some real number $\epsilon>0$.  Let $A_j$ be the function supplied by Lemma \ref{lemma_main}.  Let $r$ be a positive integer so large that $\frac{A_j(M)}{r}<\epsilon$.  We now claim that $A\subseteq U^r$, and thus has the relative property (OB).

To see this, let $f\in A$ be arbitrary.  Set $F=\phi_j(f)$, and for each $0\leq i\leq r$, let $F_i=\frac{i}{r}\cdot F$.  Note that $\|F_i\|\leq M$ for each $i$, and $\|F_i-F_{i-1}\|\leq\frac{1}{r}$.  Now for each $1\leq i\leq r$, let $f_i=\Phi_j^{-1}(F_i,\phi_{j-1}(f)(0),...,\phi_2(f)(0))$.  Informally speaking, each $f_i$ has the same ``initial conditions'' as $f$, but $\phi_j(f_i)=\frac{i}{r}\phi_j(f)$ for each $i$.  Note that $d_j(f_i,f_{i-1})\leq\frac{1}{r}$ by construction.  Therefore since $\|\Phi_j(f_i)\|\leq\|\Phi_j(f)\|\leq M$, we have\\

\begin{center} $d_j(e,f_if_{i-1}^{-1})\leq A_j(M)\cdot d_j(f_i,f_{i-1})\leq\frac{A_j(M)}{r}<\epsilon$
\end{center}
\vspace{.3cm}

\noindent by Lemma \ref{lemma_main}.  Thus if we set $g_i=f_if_{i-1}^{-1}$ for each $1\leq i\leq r$, then we have $g_i\in B_{d_j}(e,\epsilon)=U$.  Moreover $f=g_rg_{r-1}...g_2g_1$, so $f\in U^r$.  Since $f\in A$ is arbitrary, $A\subseteq U^r$ as claimed.

Now we explain how to modify the above proof for $M^1=\mathbb{S}^1$.  Again suppose that $\displaystyle\sup_{f\in A}\sup_{x\in \mathbb{S}^1}|\log f'(x)|<\infty$, and $\displaystyle\sup_{f\in A}\sup_{x\in \mathbb{S}^1}|f^{(j)}(x)|<\infty$ for every $j\leq k$.  Let us denote by $H$ the stability group of $0$ in $\Diff_+^k(\mathbb{S}^1)$.  Note that each mapping $\Phi_j$, $1\leq j\leq k$, becomes injective when restricted to $H$, and thus $d_k$ becomes a compatible metric on $H$.  Let\\

\begin{center} $A^*=\{a^*\in H:\exists t\in\mathbb{S}^1 ~ ta^*=a\}$.
\end{center}
\vspace{.3cm}

Now if $a^*\in A^*$, compute that for some rotation $t\in\mathbb{S}^1$ and for all $x\in\mathbb{S}^1$, we have $\log (a^*)'(x)=\log (at^{-1})'(x)=\log a'(t^{-1}(x))<\displaystyle\sup_{f\in A}\sup_{x\in \mathbb{S}^1}|\log f'(x)|$, and $(a^*)^{(j)}(x)=(at^{-1})^{(j)}(x)=a^{(j)}(x)\leq \displaystyle\sup_{f\in A}\sup_{x\in \mathbb{S}^1}|f^{(j)}(x)|$.  So $A^*$ satisfies the same boundedness conditions as $A$.

Let $U\subseteq\Diff_+^k(\mathbb{S}^1)$ be an open neighborhood of $e$, and let $V=U\cap H$.  $V$ contains a basic open $d_k$-ball in $H$.  Therefore, repeating the exact same argument as in the interval case, we can find a positive integer $r$ so that $A^{*}\subseteq V^r\subseteq U^r$.  Since $U$ was arbitrary, we have shown that $A^{*}$ has the property (OB) relative to $\Diff_+^k(\mathbb{S}^1)$.  But $A\subseteq\mathbb{S}^1A^{*}$, and $\mathbb{S}^1$ is compact, therefore relatively (OB) as well.  It follows that $A$ is relatively (OB) in $\Diff_+^k(\mathbb{S}^1)$ as claimed.
\end{proof}

\section{Computing Quasi-Isometry Types}

\begin{thm}  Let either $M=I$ or $M=\mathbb{S}^1$.  Then the group $\Diff_+^k(M^1)$ has the local property (OB), and thus a well-defined, non-trivial quasi-isometry class, if and only if $k<\infty$.
\end{thm}

\begin{proof}  If $k<\infty$, then $d_k$-open balls have the property (OB) by Theorem \ref{ob_character}, and they generate $\Diff_+^k(M^1)$ since the group is connected.  Thus the group has a well-defined quasi-isometry type by the results of \cite{rosendal_2015a}.  This quasi-isometry type is also non-trivial by Theorem \ref{ob_character}, because $\Diff_+^k(M^1)$ contains maps with unboundedly large derivatives and therefore does not have the property (OB) relative to itself.

It remains only to check the case $k=\infty$.  To see that $\Diff_+^\infty(M^1)$ does not have the local property (OB), let $U$ be any basic open subset of the group.  Then $U$ has the form $U=B_{d_j}(e,\epsilon)$ for some integer $1\leq j<\infty$ and some distance $\epsilon>0$.  But $U$ contains maps with unboundedly large $k$-th derivatives for $k>j$.  It follows from Theorem \ref{ob_character} that $U$ does not have property (OB) relative to $\Diff_+^\infty(M^1)$.
\end{proof}

\begin{thm}  $\Diff_+^1(I)$ is quasi-isometric to $C[0,1]^\#$ (viewed as a Banach space with the norm metric) via the mapping $\Phi_1$.
\end{thm}

\begin{proof}  It suffices to note that the metric $d_1$ is right-invariant, coarsely proper by Theorem \ref{ob_character},  isometric to the norm-metric on $C[0,1]^\#$ (via $\Phi_1$), and therefore makes $\Diff_+^1(I)$ into a geodesic space, since $C[0,1]^\#$ together with its norm metric is geodesic.  It follows from Theorem \ref{prop_47} that $d_1$ is a representative of the quasi-isometry class of $\Diff_+^1(I)$.  Therefore the right-invariant word metrics on $\Diff_+^1(I)$ are each quasi-isometric to the norm-metric on $C[0,1]^\#$.
\end{proof}

\begin{thm}  $\Diff_+^1(\mathbb{S}^1)$ is quasi-isometric to $Z=\{f\in C[0,1]^\#:f(1)=0\}$ via the mapping $\Phi_1$.
\end{thm}

\begin{proof}  Define a metric $\sigma$ on $\Diff_+^1(\mathbb{S}^1)$ by the rule\\

\begin{center} $\sigma_1(f,g)=\displaystyle\sup_{x\in\mathbb{S}^1}d_{\mathbb{S}^1}(f(x),g(x))+\|\Phi_1(f)-\Phi_1(g)\|$.
\end{center}
\vspace{.3cm}

It is straightforward to check that $\sigma_1$ is a right-invariant metric compatible with the topology on $\Diff_+^k(\mathbb{S}^1)$.  It is also coarsely proper by Theorem \ref{ob_character}.

Let $H$ denote the stabilizer of $0$ in $\Diff_+^k(\mathbb{S}^1)$.  Note that when $\sigma_1$ is restricted to $H$, we have $\sigma_1-2\leq d_1\leq \sigma_1$.  Therefore the mapping $\Phi_1$, restricted to $H$, becomes a surjective quasi-isometry from $(H,\sigma_1)$ onto $Z$ with its norm distance.  Since $Z$ is a Banach space, it is also a geodesic space, and therefore $(H,\sigma_1)$ is large-scale geodesic with some constant $K_0$.  We assume $K_0\geq 2$.

We claim $(\Diff_+^1(\mathbb{S}^1),\sigma_1)$ is large-scale geodesic with constant $K=10\pi\cdot K_0$.  By the right-invariance of $\sigma_1$, it suffices to check the large scale geodecity condition at the identity; so let $f\in\Diff_+^1(\mathbb{S}^1)$ be arbitrary.  We may also assume $\sigma_1(e,f)>10\pi\cdot K_0$.  Let $R\in\mathbb{S}^1$ be the unique rotation satisfying $R(0)=f(0)$.  So $R^{-1}f$ lies in $H$.

By the large-scale geodecity of $(H,\sigma_1)$ we may find finitely many maps $h_0,h_1,...,h_n\in H$ so that $h_0=e$, $h_n=R^{-1}f$, and $\sigma_1(h_{i-1},h_i)\leq K_0$ for $1\leq i\leq n$, and $\displaystyle\sum_{i=1}^n\sigma_1(h_{i-1},h_i)=K_0\cdot\sigma_1(e,R^{-1}f)$.

Let now $R_1$ be a rotational $n$-th root of $R$, and set $R_i=(R_1)^i$ for $0\leq i\leq n$.  Let $\ell_i=R_ih_i$ for $0\leq i\leq n$.  So $\ell_0=e$ and $\ell_n=f$.  Compute that for any rotation $S$ and any map $h\in\Diff_+^1(\mathbb{S}^1)$, we have $\Phi_1(Sh)=\Phi_1(h)$, and therefore $\sigma_1(Sh,h)=\displaystyle\sup_{x\in\mathbb{S}^1}d_{\mathbb{S}^1}(S(h(x)),h(x))\leq d_{\mathbb{S}^1}(S(0),0)\leq 2$.  Therefore

\begin{align*}
\sigma_1(\ell_{i-1},\ell_i) &= \sigma_1(R_{i-1}h_{i-1},R_ih_i)\\
&\leq \sigma_1(R_{i-1}h_{i-1},h_{i-1})+\sigma(h_{i-1},h_i)+\sigma_1(h_i,R_ih_i)\\
&\leq 2+K_0+2\leq 10\pi\cdot K_0 = K.
\end{align*}
\vspace{.3cm}

Also

\begin{align*}
\displaystyle\sum_{i=1}^n \sigma_1(\ell_{i-1},\ell_i) &\leq \sum_{i=1}^n(\sigma_1(R_{i-1}h_{i-1},h_{i-1})+\sigma(h_{i-1},h_i)+\sigma_1(h_i,R_ih_i))\\
&\leq K_0\cdot\sigma_1(e,R^{-1}f)+\displaystyle\sum_{i=1}^n(\sigma_1(R_{i-1}h_{i-1},h_{i-1})+\sigma_1(h_i,R_ih_i))\\
&\leq K_0\cdot\sigma_1(e,f)+K_0\cdot\sigma_1(f,R^{-1}f)+2\displaystyle\sum_{i=1}^n(2\pi/n+2\pi/n)\\
&\leq K_0\sigma_1(e,f)+10\pi\\
&\leq 2K_0\sigma_1(e,f) \leq 10\pi K_0\cdot\sigma_1(e,f).
\end{align*}
\vspace{.3cm}

This proves the claim.  So $\sigma_1$ is a representative of the quasi-isometry class of $\Diff_+^k(\mathbb{S}^1)$ by Theorem \ref{prop_47}.  To finish the proof, we reiterate our observation that $\Phi_1$ is a quasi-isometry between $(\Diff_+^1(\mathbb{S}^1),\sigma_1)$ and $Z$.
\end{proof}

\begin{cor} \label{cor_main}  The following are mutually quasi-isometric:
\begin{enumerate}
		\item $\Diff_+^1(I)$ as a topological group;
		\item $\Diff_+^1(\mathbb{S}^1)$ as a topological group;
		\item $C[0,1]$ as a Banach space;
		\item $C[0,1]$ as an additive topological group.
\end{enumerate}
\end{cor}

\begin{proof}  It is well-known that $C[0,1]$ is isomorphic as a Banach space to each of its closed hyperplanes (its closed subspaces of finite codimension).  This appears for instance as Exercise 5.33 in \cite{fabian_etal_2001a}.  It follows that $C[0,1]$, $C[0,1]^\#$, and $Z=\{F\in C[0,1]:F(1)=0\}$ are mutually quasi-isometric.  This proves the quasi-isometric equivalence of (1)--(3).  The fact that (3) and (4) are quasi-isometric is a special case of a general fact about Banach spaces proved in \cite{rosendal_2015a}.
\end{proof}




We have proven that $\Diff_+^k(M^1)$ is quasi-isometric to $C[0,1]^\#$, a closed hyperplane of $C[0,1]$, for either choice of $M^1=I$ or $M^1=\mathbb{S}^1$.  Referring back to the road map in Figure 1, it seems natural to conjecture that $\Diff_+^2(M^1)$ is quasi-isometric to $C[0,1]$; $\Diff_+^3(M^1)$ is quasi-isometric to $C[0,1]\oplus\mathbb{R}$; etc.  It would then follow that all of these topological groups are mutually quasi-isometric, since the corresponding Banach spaces are isomorphic.  If this is the case, then evidently computing quasi-isometry types fails to distinguish the groups $\Diff_+^k(M^1)$.

\begin{question}  Is $\Diff_+^k(M^1)$ quasi-isometric to $\Diff_+^j(M^1)$ for $j\neq k$?
\end{question}


\begin{thebibliography}{99}

\bibitem{banach_mazur_1932a} S. Banach,
{\it Th\'{e}orie des op\'{e}rations lin\'{e}aires.}  Monografie Matematyczne, Warsaw, 1932.

\bibitem{decornulier_delaharpe_2015a} Y. de Cornulier and P. de la Harpe,
{\it Metric geometry of locally compact groups}, preprint.

\bibitem{fabian_etal_2001a} Fabian, M., Habala, P., Hajek, P., Montesinos Santalucia, V., Pelant, J., Zizler, V.,
{\it Functional Analysis and Infinite-Dimensional Geometry.}  Springer-Verlag, New York, 2001.

\bibitem{kalton_2008a} Kalton, N. J.,
{\it The Nonlinear Geometry of Banach spaces}, Revista Matem\'{a}tica Complutense 21 (2008), 7–60.

\bibitem{ostrovskii_2013a} Ostrovskii, M,
{\it Metric embeddings: bilipschitz and coarse embedddings into Banach spaces,} de Gruyter Studies in Mathematics, vol. 49, de Gruyter, Berlin, 2013.

\bibitem{rosendal_2009a} C. Rosendal,
{\it A topological version of the Bergman property}. Forum Mathematicum 21 (2009), no. 2, 299-332.

\bibitem{rosendal_2015a} C. Rosendal,
{\it Coarse geometry of topological groups}, preprint.

\bibitem{rosendal_2015b} C. Rosendal,
{\it Large scale geometry of automorphism groups}, preprint.

\end{thebibliography}
\end{document}